\documentclass[11pt]{amsart}
\usepackage{a4wide,amssymb}
\RequirePackage{hyperref}

\newcommand{\E}{\mathcal E}
\newcommand{\w}{\omega}

\theoremstyle{plain}
\newtheorem{theorem}{Theorem}

\newtheorem{example}{Example}
\newtheorem{question}{Question}
\theoremstyle{definition}
\newtheorem{remark}{Remark}

\begin{document}

\title{Minmax bornologies}

\author{Taras Banakh, Igor Protasov}
\address{T.Banakh: Ivan Franko National University of Lviv, Universytetska 1, 79000, Lviv, Ukraine}
\email{t.o.banakh@gmail.com}
\address{I.Protasov: Faculty of Computer Science and Cybernetics, Kyiv University,         Academic Glushkov pr. 4d, 03680, Kyiv, Ukraine}
\email{i.v.protasov@gmail.com}

\vskip 5pt

\begin{abstract}
A bornology  $\mathcal{B}$ on a set $X$  is called minmax if the smallest and the largest coarse structures on $X$  compatible with $\mathcal{B}$  coincide.
  We prove that $\mathcal{B}$ is  minmax if and only if the family $\mathcal B^\sharp=\{p\in\beta X:\{X\setminus B:B\in\mathcal B\}\subset p\}$ consists of ultrafilters which are pairwise non-isomorphic via $\mathcal B$-preserving bijections  of  $X$. Also we construct a minmax bornology $\mathcal B$ on $\w$ such that the set $\mathcal B^\sharp$ is infinite. We deduce this result from the existence of a  closed infinite subset in $\beta\w$ that consists of pairwise non-isomorphic ultrafilters. 
\end{abstract}
\maketitle

\section{Introduction}

Let $X$  be a set. A family $\mathcal{E}$ of subsets of $X\times X$ is called a {\it coarse structure } if

\begin{itemize}
\item  each $E\in \mathcal{E}$  contains the diagonal  $\bigtriangleup _{X}:= \{(x,x): x\in X\}$ of $X$;
\item  if  $E$, $E^{\prime} \in \mathcal{E}$ then $E\circ E^{\prime}\in\mathcal{E}$ and
$E^{-1}\in \mathcal{E}$,   where    $E\circ E^{\prime}=\{(x,y): \exists z((x,z) \in  E,  \   \ (z, y)\in E^{\prime})\}$ and   $E^{-1}=\{(y,x): (x,y)\in E\}$;
\item if $E\in\mathcal{E}$ and $\bigtriangleup_{X}\subseteq E^{\prime}\subseteq E  $   then
$E^{\prime}\in \mathcal{E}$;
\item  $\bigcup\mathcal E=X\times X$.
\end{itemize}
\vskip 7pt

A subfamily $\mathcal{E}^{\prime} \subseteq \mathcal{E}$  is called a
{\it base} for $\mathcal{E}$  if for every $E\in \mathcal{E}$ there exists
  $E^{\prime}\in \mathcal{E}^{\prime}$  such  that
  $E\subseteq E ^{\prime}$.
For $x\in X$,  $A\subseteq  X$  and
$E\in \mathcal{E}$, we denote
$E[x]= \{y\in X: (x,y) \in E\}$,
 $E [A] = \bigcup_{a\in A}   \   \   E[a]$
 and say that  $E[x]$
  and $E[A]$
   are {\it balls of radius $E$
   around} $x$  and $A$.

The pair $(X,\mathcal{E})$ is called a {\it coarse space} \cite{b4}  or a {\em ballean} \cite{b2}, \cite{b3}.

For a coarse space $(X,\E)$, a subset $B\subset X$ is called {\it bounded} if $B\subset E[x]$  for some $E\in \mathcal{E}$ and $x\in X$.
A coarse space $(X,\mathcal{E})$ is  called {\it  unbounded} if $X$ is unbounded.
In what follows,  all
 balleans under consideration are {\bf supposed to
 be unbounded}.

A family $\mathcal{B}$ of subsets of an infinite set $X$ is called a {\it bornology} on $X$ if  $\cup\mathcal{B}=X \notin \mathcal{B}$ and $\mathcal{B}$ is
closed under taking subsets and
finite  unions.
For a coarse space $(X, \mathcal{E})$,
we denote by
$\mathcal{B}_{(X, \mathcal{E})}$
the bornology
  of all bounded subsets of $(X, \mathcal{E})$.

A  coarse structure $(X, \mathcal{E})$ is called
\vskip 5pt

\begin{itemize}
\item {\it discrete}  if for every $E \in \mathcal{E}$   there exists
  $B\in \mathcal{B}_{(X, \mathcal{E})} $
   such that  $E [x]=\{ x\}$  for each
$x\in X \backslash B$;
\item{}  {\it ultradiscrete }  if
$(X, \mathcal{E})$
 is discrete and the  family
$\{ X \backslash B:  B \in \mathcal{B}_{(X, \mathcal{E})}\}$ is an ultrafilter;
\item{}   {\it  maximal}  if   $(X, \mathcal{E})$ is bounded  in every strictly stronger coarse structure  on $X$.
\end{itemize}

\vskip 7pt

Let $\mathcal{B}$  be a bornology on   $X$.
Following \cite{b1},  we say that a coarse  structure $\mathcal{E}$
on $X$  is {\it compatible}   with  $\mathcal{B}$  if
 $\mathcal{B}= \mathcal{B}_{(X, \mathcal{E})}$.

By \cite[\S6]{b1}, the family of all coarse structures  on $X$,  compatible  with  $\mathcal{B}$,   has  the smallest and the largest   elements  ${\Downarrow}\mathcal B$  and  ${\Uparrow}\mathcal B$, respectively.

The smallest coarse structure ${\Downarrow} \mathcal{B}$, compatible with the bornology $\mathcal B$,  is generated by the base consisting of the  entourages
$(B\times  B)\cup\bigtriangleup_{X}$,
 where  $B\in \mathcal{B}$.
A coarse structure $\mathcal{E}$  on   $X$  is  discrete  if and only if
$(X, \mathcal{B})= {\Downarrow}\mathcal{B}_{(X, \mathcal{E})}$.
A discrete coarse  structure is  maximal  if and  only if $\mathcal{E}$  is ultradiscrete \cite[Example 10.1.2]{b3}.

The largset coarse structure ${\Uparrow}\mathcal B$, compatible with the bornology $\mathcal B$, consists of all entourages $E\subset X\times X$ such that for any set $B\in\mathcal B$ the set $E[B]\cup E^{-1}[B]$ belongs to $\mathcal B$, see \cite[\S 6]{b1}.

\section{Characterizing minmax bornologies}

A bornology  $\mathcal{B}$ on a set $X$  is called {\it   minmax} if  ${\Downarrow} \mathcal B={\Uparrow}\mathcal B$.
Equivalently,  $\mathcal{B}$   is  minmax  if $\mathcal B$ is compatible with a unique coarse structure on $X$. It is clear that each ultradiscrete bornology is minmax. In this section we show that the converse is not true.


We recall that two ultrafilters  $p, q$  on a set $X$  are {\it isomorphic}  if there exists a bijection
 $f: X\to  X$  such that the ultrafilter $\bar f(p):=\{f(P):P\in p\}$ is equal to $q$.
 
Let $\mathcal B$ be a bornology on a set $X$. We say that two ultrafilters $p, q$ on $X$  are  $\mathcal{B}$-{\it isomorphic} if there is a bijection $f:X\to X$ such that
$f(p)=q$ and $\{f(B):B\in\mathcal B\}=\mathcal B$.

We denote by  $\mathcal{B}^{\sharp}$  the set of all ultrafilters  $p$  on $X$  such  that $\{X\setminus B:B\in\mathcal B\} \subset p$. Observe that a bornology $\mathcal B$ is ultradiscrete if and only if $\mathcal B^\sharp$ is a singleton.

\begin{theorem}\label{t:main}
A bornology  $\mathcal{B}$   on a  set $X$  is  minmax  if and only if every   two  distinct   ultrafilters   $p, q\in \mathcal{B}^\sharp$   are   not  $\mathcal{B}$-isomorphic.
\end{theorem}

\begin{proof}  To prove the ``only if'' part, assume that there exist two distinct  $\mathcal{B}$-isomorphic  ultrafilters  $p, q $  in  $\mathcal{B}$  and take a bijection  $h:  X\to X$  witnessing  this fact. Since  $p\neq  q$,   the set  $\{x \in X:  f(x)\neq  x \}$ does not belong to the bornology $\mathcal B$.  Then the entourage $E=\{(x,y):x\in X:y\in\{x,f(x)\}\}$ belongs to the coarse structure ${\Uparrow}\mathcal B$ and witnesses that it is not discrete and hence not equal to ${\Downarrow}\mathcal B$.
This means that $\mathcal B$ is not minmax.
\smallskip

To prove the ``if'' part, assume that $\mathcal{B}$  is not minmax.
Then the coarse structure ${\Uparrow} \mathcal{B}$
  is not discrete  and there exists
  $E\in {\Uparrow} \mathcal{B}$
     such that the set
   $\{x \in X: |E[x]|>  1 \}$
     does not belong to the bornology $\mathcal B$. 
We  take a maximal by inclusion  subset
 $Y \subseteq \{x \in X: |E[x]|>  1 \}$
 such that
   $E[y]\cap  E[z] = \emptyset$ for all  distinct $y, z\in Y$.
 We note that $Y$ does not belong to $\mathcal B$ and
  and take an arbitrary  ultrafilter
  $p\in \mathcal{B}^{\sharp}$
       such that  $Y\in p$.
For each  $y\in Y$ choose $z _{y}\in E [y]\setminus y$ and consider  the  bijection
 $f: X \to X$  acting as the transposition on  each pair  $y, z _{y}$  and   identical on all other  elements  of $X$.
Observe that $\{f(B):B\in\mathcal B\}=\mathcal B$ and hence $p\ne \bar f(p)\in\mathcal B^\sharp$. So $p$ and $\bar f(p)$ are two distinct $\mathcal B$-isomorphic ultrafilters in $\mathcal B^\sharp$.
\end{proof}

The following example was first presented in \cite[Example 1]{b3}.

\begin{example} There exists a minmax bornology $\mathcal B$ on $\w$ which is not ultradiscrete.
\end{example}

\begin{proof} Choose any two non-isomorphic ultrafilters $p,q$ on $\w$ and consider the bornology $\mathcal B=\{B\subset \w:\w\setminus B\in p\cap q\}$. Since $\mathcal B^\sharp=\{p,q\}$, the bornology $\mathcal B$ is minmax (by Theorem~\ref{t:main}) and not ultradiscrete.
\end{proof}

Now we shall construct a minmax bornology $\mathcal B$ on $\w$ for which the set $\mathcal B^\sharp$ has cardinality $2^{\mathfrak c}$. For this we need the following fact, which can have an independent value.

\begin{theorem}\label{l:main} The Stone-\v Cech compactification $\beta\omega$ of $\w$ contains a closed infinite subset consisting of pairwise non-isomorphic ultrafilters.
\end{theorem}

\begin{proof}  Let us recall that a point $x$ of a topological space $X$ is called a {\em weak $P$-point} if $x$ does not belong to the closure of any countable subset $C\subset X\setminus\{x\}$. By Corollaries 4.5.2 and 4.3.2 in \cite{vM}, the space $\beta \omega\setminus\w$ contains $2^{\mathfrak c}$ weak $P$-points. Consequently, we can choose a sequence of free ultrafilters $(p_n)_{n\in\w}$ consisting of pairwise non-isomorphic weak $P$-points in $\beta\w\setminus\w$. The definition of a weak $P$-point implies that the subspace $D=\{p_n\}_{n\in\w}$ of $\beta\w\setminus \w$ is discrete. Now the regularity of $\beta\w$ implies that there exists a family $\{P_n\}_{n\in\w}$ of pairwise disjoint sets in $\w$ such that $P_n\in p_n$ for every $n\in\w$. 

We claim that the closure $\bar D$ of $D$ consists of pairwise non-isomorphic ultrafilters.
To derive a contradiction, assume that $\bar D$ contains two distinct isomorphic ultrafilters $p,q$. Then $p\in\bar P$ and $q\in\bar Q$ for some disjoint sets $P,Q\subset D$. Find a bijection $f:\w\to\w$ such that $f(p)=q$. The bijection $f$ extends to a homeomorphism $\bar f:\beta\w\to\beta\w$. Since the set $D$ consists of pairwise non-isomorphic ultrafilters, $\bar f(P)$ is disjoint with $Q$. So, $\bar f(P)$ and $Q$ are two disjoint countable discrete subspaces of $\beta\w\setminus\w$ consisting of weak $P$-points. Then $\bar f(P)$ is disjoint with the closure of $Q$ and $Q$ is disjoint with the closure of $\bar f(P)$ (which is equal to $\bar f(\bar P)$).
By Lemma 1 of Frol\'{\i}k \cite{Frolik} (see also Theorem 1.5.2 in \cite{vM}),  $\bar Q\cap \bar f(\bar P)=\emptyset$. On the other hand, $q=\bar f(p)\in\bar{Q}\cap\bar f(\bar P)$.
\end{proof}

\begin{remark} Theorem~\ref{l:main} has also a ``near-coherent'' version. Let us recall \cite{Blass86} that two ultrafilters $p,q$ on $\w$ are {\em near-coherent} if there exists a finite-to-one function $f:\w\to\w$ such that $\bar f(p)=\bar f(q)$. It is clear that any two isomorphic ultrafilters on $\w$ are near-coherent. By \cite{BB}, the space $\beta\w$ contains an infinite closed set consisting of pairwise non-near-coherent ultrafilters if and only if $\beta\w$ contains infinitely many non-near-coherent ultrafilters. The latter happens if $\mathfrak u\ge\mathfrak d$. On the other hand, by \cite[9.18]{Blass}, the strict inequality $\mathfrak u<\mathfrak g$ (which is consistent with ZFC by \cite[11.2]{Blass}) implies that all free ultrafilters on $\w$ are near-coherent.
\end{remark} 

\begin{example} There  exist a minmax bornology $\mathcal{B}$ on $\omega$  such that $|\mathcal B^\sharp|=2^{\mathfrak c}$.
\end{example}

\begin{proof} By Theorem~\ref{l:main}, the space $\beta\w$ contains an infinite closed subset $F$ consisting of pairwise non-isomorphic ultrafilters. By Lemma 3.1.2(c) in \cite{vM}, $|F|=2^{\mathfrak c}$. Consider the bornology $\mathcal B=\{B\subset\w:\w\setminus B\in\bigcap_{p\in F}p\}$ and observe that $\mathcal B^\sharp=F$. By Theorem~\ref{t:main}, the bornology $\mathcal B$ is minmax.
\end{proof}

\section{Characterizing bornologies with maximal coarse structure ${\Uparrow}\mathcal B$}

By \cite[Theorem 10.2.1]{b3}, any unbounded set $L$ in a maximal coarse space $(X,\E)$  is {\em large} (which means that $X=E[L]$ for some $E\in\E$).
The converse is not true:

\begin{example} There exists a coarse structure $\E$ on a countable set $X$ such that the coarse space $(X,\E)$ is not maximal but each unbounded subset of $(X,\E)$ is large.
\end{example}

\begin{proof} Let $X=\w$, $G$ be the group of all finitely supported permutations of $X$, and $[G]^{<\w}$ be the family of all finite subsets of $G$. The action of the group $G$ induces the coarse structure $$\E=\{E\subset X\times X:\exists F\in[G]^{<\w},\;\triangle_X\subset E\subset\{(x,y):y\in\{x\}\cup Fx\}\}$$
on $X$, whose bornology coincides with the bornology $\mathcal B$ of all finite subsets of $X$.

The coarse structure $\E$ is not maximal since $\E\subset{\Uparrow}\mathcal B$ and $\E\ne{\Uparrow}\mathcal B$. Indeed, the coarse structure ${\Uparrow}\mathcal B$ contains the entourage $E=\bigcup_{n\in\w}[n^2,(n+1)^2)\times [n^2,(n+1)^2)$ that does not belong to the (finitary) coarse structure $\E$.

On the other hand, each unbounded set $L\subset X$ is large since we can find a bijection $f:X\to X$ such that $X\setminus L\subset f(L)$. This bijection determines the entourage $E:=\{(x,y)\in X\times X:y\in\{x,f(x)\}\}\in\E$ such that $E[L]=X$.
\end{proof}

\begin{theorem}\label{t:max} For a bornology $\mathcal{B}$  on a set $X$, the coarse space
 $(X,{\Uparrow}\mathcal{B})$  is maximal if and only if each unbounded subset of
 $(X, {\Uparrow}\mathcal{B})$ is large.
  \end{theorem}

\begin{proof} The ``only if'' part follows from Theorem 10.2.1 in \cite{b3}.
\smallskip

To prove the ``if'' part, assume that each unbounded subset of
$(X, {\Uparrow}\mathcal{B})$
 is large, but $(X, {\Uparrow}\mathcal{B})$
   is  not maximal.
Then there is an unbounded coarse structure $\E$  on  $X$  such  that ${\Uparrow}\mathcal{B}\subsetneq \E$. By the definition of ${\Uparrow}\mathcal B$, the coarse structure $\E$ is not compatible with the bornology $\mathcal B$. Consequently, there exists a set $L\subset X$ which is bounded in $(X,\E)$ but does not belong to $\mathcal B$. Then $L$ is unbounded in $(X,{\Uparrow}\mathcal B)$ and hence $X=E[L]$ for some $E\in{\Uparrow}\mathcal B\subset\E$, which implies that $X$ is bounded in $(X,\E)$. But this contradicts the choice of $\E$.
  \end{proof}

\begin{example}\rm Theorem~\ref{t:max} implies that for any infinite set $X$ and the bornology $\mathcal B:=\{A\subset X:|A|<|X|\}$ the coarse structure ${\Uparrow}\mathcal B$ is maximal. Indeed, for any subset $L\subset X$ of cardinality $|L|=|X|$, we can find a bijection $f$ of $X$ such that $X\setminus L\subset f(L)$. Then $E=\{(x,y)\in X\times X:y\in\{x,f(x)\}\}$ is an entourage in ${\Uparrow}\mathcal B$ such that $X=E[L]$, which means that the set $L$ in large in $(X,{\Uparrow}\mathcal B)$.
\end{example}

 \smallskip 

Following  \cite{b1},   we say that a coarse structure  $\mathcal{E}$  on $X$   is {\it relatively maximal}
 if $\E={\Uparrow} \mathcal{B}_{(X, \mathcal{E})}$.
Clearly,  $\E$  is  relatively  maximal if either $\E$  is maximal or
$\mathcal{B}_{(X, \mathcal{E})}$ is minmax.

\begin{question} Given a coarse structure $\E$,  how can one detect whether $\E$  is  relatively maximal?
\end{question}

\begin{remark} In light of  Theorem~\ref{t:main},  it  is a very rare case  when  the coarse structure
 $\mathcal{E}$ on $X$ is uniquely defined  by the   bornology
 $\mathcal{B}_{(X, \mathcal{E})}$.
 
We denote by $so (X,\mathcal{E})$ the set of all  slowly oscillating functions  of
$(X,\mathcal{E})$.
By Theorem 7.3.1 \cite{b3}, if the coarse structures $\mathcal{E}$  and $\mathcal{E}^{\prime}$  on $X$  have linearly ordered bases and
$\mathcal{B}_{(X, \mathcal{E})}=\mathcal{B}_{(X, \mathcal{E})^{\prime}}$,
$so{(X, \mathcal{E})}=so{(X, \mathcal{E}^{\prime})}$,
 then
 $\mathcal{E}=\mathcal{E}^{\prime}$.
We denote by
$\delta_{(X,\mathcal{E})}$
 the binary relation on
the power-set  $2^X$ of $X$ 
 defined by
 $A \delta_{(X,\mathcal{E})} B$
   if and only if there  exists $E\in \mathcal{E}$   such that
   $A\subseteq E[B]$ and $B\subseteq E[A]$.
By Theorem 7.5.3 from \cite{b3}, if  coarse  structures  $\E,\E'$ on $X$  have linearly ordered bases and
$\delta_{(X,\mathcal{E})}=\delta_{(X,\mathcal{E}^{\prime})}$
  then $\mathcal{E}=\mathcal{E}^{\prime}$.
\end{remark}

\end{document}